\documentclass{article}
\usepackage{amssymb}
\usepackage{amsfonts}
\usepackage{amsmath}
\usepackage{fullpage}
\usepackage{graphicx}
\usepackage{mathrsfs}

\newtheorem{theorem}{Theorem}

\newtheorem{corollary}[theorem]{Corollary}

\newtheorem{definition}[theorem]{Definition}

\newtheorem{lemma}[theorem]{Lemma}

\newtheorem{proposition}[theorem]{Proposition}

\newenvironment{proof}[1][Proof]{\noindent{\sc#1} }{\hfill \rule{0.5em}{0.5em} \\}
\newenvironment{remark}[1][Remark]{\noindent \textbf{#1. }{}}

\begin{document}

\title{
Tail generating functions for extendable branching processes 
}
\author{Serik Sagitov\thanks{%
Chalmers University and University of Gothenburg, 412 96 Gothenburg, Sweden. Email address:
serik@chalmers.se. 
}
}
\date{}
\maketitle

\begin{abstract}
We study branching processes of independently splitting particles in the continuous time setting.
If  time is calibrated such that particles live on average one unit of time, the corresponding transition rates are fully determined by the generating function $f$ for the offspring number of a single particle.
We are interested in the deffective case $f(1)=1-\epsilon$, where
each splitting particle with probability $\epsilon $ is able to
 terminate the whole branching process. 
A branching process $\{Z_t\}_{t\ge0}$ will be called {\it extendable} if $f(q)=q$ and $f(r)=r$ for some $0\le q<r<\infty$.
Specializing on the extendable case we derive an integral equation for $F_t(s)={\rm E} s^{Z_t}$. This equation is expressed in terms of what we call, {\it tail generating functions}. 
With help of this equation, 
we obtain limit theorems for the time to termination as $\epsilon \to0$. We find that conditioned on non-extinction, the typical values of  the termination time  follow an exponential distribution in the nearly subcritical case, and require different scalings depending on whether the reproduction regime is asymptotically 
critical or supercritical.
Using the tail generating function approach we also obtain new refined asymptotic results for the regular branching processes with  $f(1)=1$.

\end{abstract}

\section{Introduction}

We consider a single type Markov branching process  $\{Z_t\}_{t\ge0}$ with continuous time, assuming $Z_0=1$. This is a basic stochastic model for a population of particles having exponential life lengths with a parameter $\lambda$. It is thought that each particle at the moment of death is replaced by  a random number of offspring particles according to a common reproduction law having a probability generating function
\[f(s)=\sum_{k=0}^\infty s^kp_k,\quad s\in[0,1].\]
Without loss of generality, we will always assume that  $\lambda=1$. 
To recover a general $\lambda$ case from our results, one should just replace the time variable $t$ by $\lambda t$. We also exclude the trivial case $p_1=1$.



Under the natural assumption $f(1)=1$, the population mean formula ${\rm E}(Z_t)=e^{(m_1-1)t}$, where $m_1=f'(1)$, identifies three different regimes of reproduction: subcritical, critical, and supercritical, depending on whether  $m_1$ is smaller, equal, or larger than 1.  
The probability of ultimate extinction of the branching process, $q={\rm P}(Z_\infty=0)$, equals 1 in the subcritical and critical cases, and in the supercritical case it is given by  the smallest non-negative root of the equation $f(x)=x$, see for example  \cite{HJV, KA}.

In this paper  we allow for defective probability distributions by letting $f(1)<1$. In this case, each particle 
with  probability $1-f(1)$ 
%
 sends the Markov process $\{Z_t\}$ to an absorbing graveyard state $\Delta$.
 Such non-regular branching processes has got a limited attention in the literature. In \cite{KT}, this setting for the linear birth-death processes was interpreted as a population model with killing.
A related account on a special class of branching processes allowing for explosive particles is given in \cite{SL}. 
Another, biologically relevant interpretation for the termination event is favourable mutation. 
Think of a branching process governed by a subcritical reproduction law $g$ with $g(1)=1$ and $g'(1)<1$, which may escape extinction due to a mutation switching the reproduction rate for the new type of particles into a supercritical regime \cite{SS}. Such a process stopped at the first mutation event, can be modelled by a single type branching process with
$f(s)=g(s(1-\mu))$,
where 
$\mu$ is the probability for a particle to mutate at birth. If $\mu$ is small, then 
$1-f(1)=1-g(1-\mu)\sim \mu g'(1)$.

Another non-regular case, not addressed here, is that of explosive branching processes with $f(1)=1$. 
The interested reader is referred to \cite{Pa1} investigating a broad class of such non-regular Markov processes.

\begin{definition}
 For $0\le q\le1\le r<\infty$ and $q<r$, we say that a possibly deffective probability generating function $f$ is $(q,r)$-extendable, if $f(q)=q$ and $f(r)=r$. 
 A branching process whose reproduction law has a $(q,r)$-extendable generating function will be called a $(q,r)$-extendable branching process. 
 \end{definition}

Figure \ref{f1} depicts a graph for a $(q,r)$-extendable probability generating function with $f(1)<1$. 
The focus of this paper is on the $(q,r)$-extendable branching processes. In particular, when $f(1)=1$, our results apply to the extendable subcritical case with $q=1<r<\infty$ and $0<m_1<1<f'(r)$, as well as to the supercritical case with $0\le q<1=r$ and $0\le f'(q)<1<m_1<\infty$. The subcritical extendable branching processes arise naturally as supercritical branching processes conditioned on extinction, see \cite{GP} and \cite{JL}. 

\begin{figure}
\centering
\includegraphics[width=8cm]{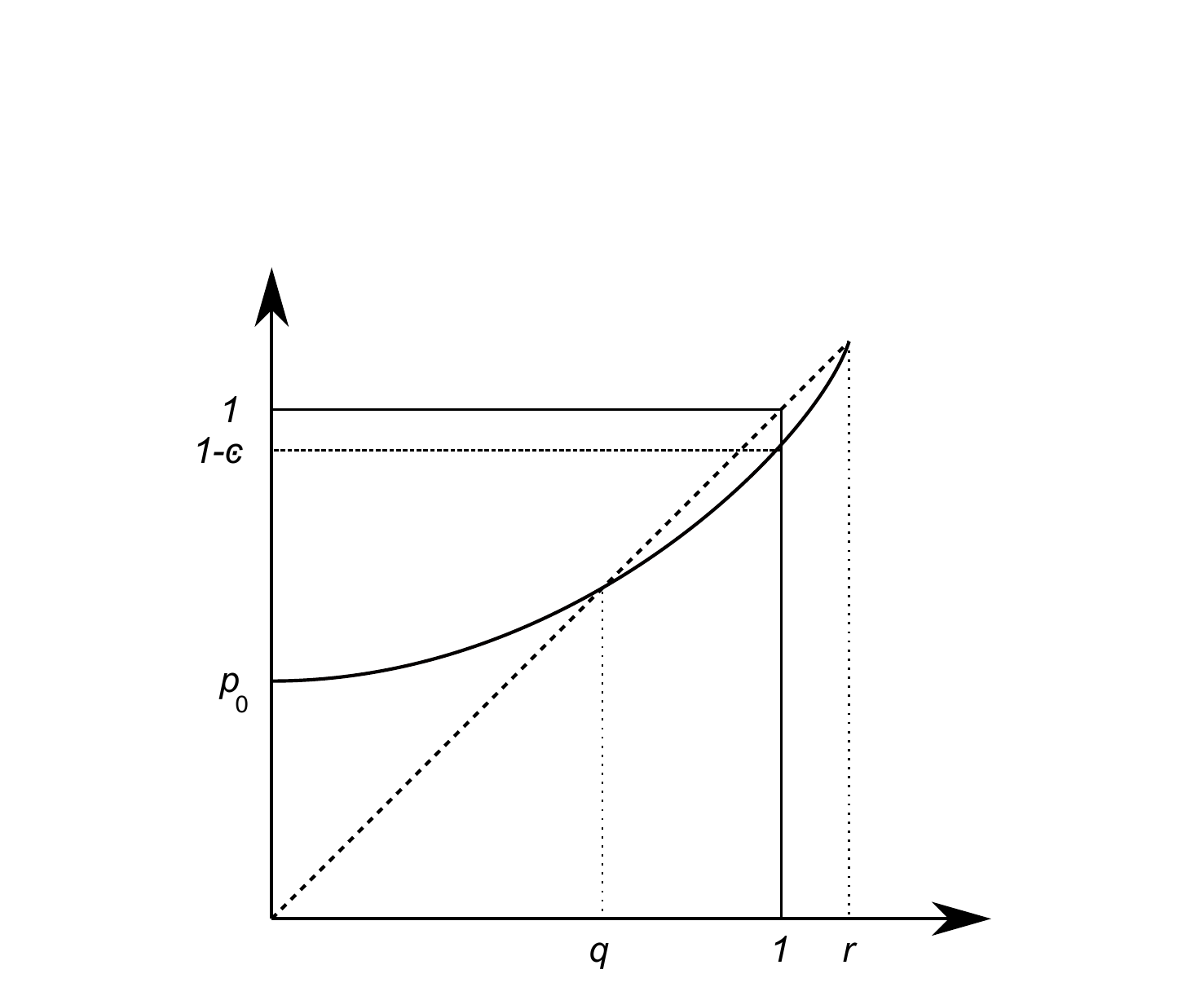}
 \caption{A $(q,r)$-extendable probability generating function $f(s)$.}
 \label{f1}
\end{figure}

Theorem \ref{main} below, proposes a new form of  the backward Kolmogorov equation valid for the $(q,r)$-extendable branching processes. It is expressed in terms of what we call {\it tail generating functions}. The name comes from a simple observation involving the tail probabilities of the reproduction law:
 \[{f(1)-f(s)\over1-s}=\sum_{i\ge0} s^i\sum_{j\ge i+1} p_j.\]
 The last generating function will be denoted $f^{(2)}(1,s)$ for $s\in[0,1)$, and also by continuity, we will put $f^{(2)}(1,1)=f'(1)=m_1$.
\begin{definition} \label{def}
Given a power series 
$ v(s)=\sum_{k=0}^\infty s^kv_k$ with all $v_k\ge0$,
define its $n$-th order tail generating function by
\begin{align*}
&v^{(n)}(s_1,\ldots,s_{n})
=\sum_{i_1\ge0,\ldots,i_n\ge0}v_{j_n}s_1^{i_1}\ldots s_{n}^{i_{n}},
\end{align*}
where $j_n=i_1+\ldots+i_n+n-1$, $n\ge1$.
\end{definition}

Observe that the tail generating functions are symmetric functions which can be computed using a simple recursion
\begin{align}\label{lay}
 v^{(n)}(s_1,\ldots,s_n)&={v^{(n-1)}(s_1,\ldots,s_{n-1})-v^{(n-1)}(s_2,\ldots,s_n)\over s_1-s_n},\quad s_1\neq s_n.
\end{align}
Whenever some of the arguments coincide,  the following rule applies, see Section \ref{Sta}, 
\begin{align}\label{deri}
v^{(k+n)}(s_1,\ldots,s_{k-1},s,\ldots,s)={1\over n!}{d^n\over ds^n}v^{(k)}(s_{1},\ldots,s_{k-1}, s), \quad k\ge 1, n\ge1.
\end{align}
In particular, the second order tail generating functions satisfy
$$v^{(2)}(s_1,s_2)={v(s_1)-v(s_2)\over s_1-s_2}, \quad s_1\neq s_2,\qquad v^{(2)}(s,s)=v'(s),$$
so that for a given $(q,r)$-extendable probability generating function $f$, we have
$$f^{(2)}(q,s)={f(s)-q\over s-q}, \quad f^{(2)}(q,q)=f'(q), \quad f^{(2)}(q,r)=1,\quad f^{(2)}(r,s)={r-f(s)\over r-s}, \quad f^{(2)}(r,r)=f'(r).$$
 
 \begin{theorem}\label{main}
Let $\{Z_t\}$ be a $(q,r)$-extendable branching process.  
 Then for $t>0$, the probability generating function $F_t(s)={\rm E}(s^{Z_t})$ is $(q,r)$-extendable. If $f'(r)<\infty$, then 
 $$F'_t(q)=e^{- \alpha t},\quad \alpha=1-f'(q),\qquad F'_t(r)=e^{\beta t},\quad \beta=f'(r)-1,$$ 
$\gamma={\alpha/\beta}\in(0,1]$, and for $s\in[0,q)\cup(q,r)$,
\begin{align}\label{equ}
{F_t(s)-q\over s-q}= e^{- \alpha t}\Big\{{r-F_t(s)\over r-s}\Big\}^\gamma \exp\Big\{\int_{s}^{F_t(s)}\psi_{q,r}(x)dx\Big\},
 \end{align}
where 
 \begin{align}\label{psih}
 \psi_{q,r}(x)={f^{(4)}(q,q,r,x)-\gamma f^{(4)}(q,r,r,x)\over  f^{(3)}(q,r,x)}.
 \end{align}

\end{theorem}

The integrand $ \psi_{q,r}(x)$ appearing in Theorem \ref{main} has no singularities over the interval $x\in[0,r)$, and
\begin{equation}\label{qqrr}
  \psi_{q,r}(q)={\gamma\over r-q}-{f''(q)\over 2},\quad \psi_{q,r}(r)={1\over r-q}-{\gamma f''(r)\over 2\beta}.
\end{equation}
Theorem \ref{main}  is proved in Section \ref{proo}. Section \ref{proco} contains Theorem \ref{main0} addressing the critical case, which corresponds to the parameter option  $(q,r)=(1,1)$ excluded from Theorem \ref{main}. Theorem \ref{main0} proposes equation \eqref{eqc} as a 
counterpart of \eqref{equ} for the critical branching processes. Section \ref{Ex2} discusses an important special case of equations \eqref{equ} and \eqref{eqc} where the integral parts vanish due to $\psi_{q,r}(x)\equiv 0$. Condition $\psi_{q,r}(x)\equiv 0$ leads to a four-parameter family of  possibly deffective probability distributions. These, what we call, {\it modified linear-fractional distributions}, have interest of its own as an extension of the well-known family of linear-fractional distributions.
Another illuminating case, where the integrals in 
equations \eqref{equ} and \eqref{eqc}  are computed explicitly, is presented  in Section \ref{Ex}. 

In Section \ref{Sta} we  analyse finiteness of the integrals $\int_{0}^{r}\psi_{q,r}(x)dx$ connected to Theorems \ref{main} and \ref{main0}. We find that the $x\log x$-type conditions playing a  crucial role in the theory of branching processes \cite{AN}, are expressed naturally in terms of tail generating functions.
In Section \ref{Ssub} we apply the tail generating function approach to obtain a novel Yaglom type theorem for extendable branching processes conditioned on $0<Z_t<\infty$.

Note that branching processes with $f(1)<1$ fall outside the usual classification system, as irrespective of the value of $m_1$, the probability of ultimate extinction $q$ is always less that one, see  Section \ref{proo}. In Section \ref{ex3} we consider a family of $(q_\epsilon,r_\epsilon)$-extendable branching processes such that
 for some $s_0>1$,
\begin{equation}\label{epo}
 f_\epsilon(s)\to f(s),\quad f(1)=1,\quad f(s_0)<\infty,\quad f(s_0)\neq s_0,\quad s\in[0,s_0],
\end{equation}
 as $\epsilon\to 0$. 
 In this setting we can speak of nearly subcritical, critical, and supercritical 
extendable branching processes. We study the distribution of the  termination time conditioned on non-extinction, and  conclude that 
the largest values of  the termination time (proportional to  $1/\sqrt \epsilon$) are expected in the balanced nearly critical case.

Finally, in Sections \ref{Scri} - \ref{Ssup} we apply the tail generating function approach to the regular case $f(1)=1$. Using Theorems \ref{main}, \ref{main0}, and results from Section \ref{Sta} we obtain a new refined asymptotic formula for critical branching processes, and then give streamlined proofs for the known facts in the supercritical case.

\section{Proof of Theorem \ref{main} }\label{proo}
%
%
%


\begin{lemma}\label{ieq1}
 Consider a branching process with  $f(1)\le1$. Its probability of extinction
$q={\rm P}(Z_\infty=0)$ is the smallest non-negative root of $f(x)=x$. 
The (possibly deffective) probability generating functions $x_t=F_t(s)$ of the branching process satisfy the backward Kolmogorov equation
 \begin{equation}\label{ode}
{d x_t\over d t}=f(x_t)-x_t,\qquad x_0=s,\qquad s\in[0,1).
\end{equation}

\end{lemma}
\begin{proof}
Let $L$ and  $X$ be the life length and offspring number of the ancestral particle. In the deffective case with $f(1)<1$ we assume that $1-f(1)={\rm P}(X=\Delta)$ and  $s^\Delta=0$.
Then by the branching property,
 \[Z_t=1_{\{L>t\}}+1_{\{L\le t\}}1_{\{X\ne \Delta\}}\sum_{i=1}^{X}Z_{t-L}^{(i)}+1_{\{L\le t\}}1_{\{X= \Delta\}}\cdot\Delta,\]
 where  $Z_{u}^{(i)}$ stands for the branching process stemming from the $i$-th ancestral daughter. 
 This yields in term of generating functions
 \[F_t(s)=se^{- t}+\int_0^tf(F_{t-u}(s))e^{- u}du,\]
 due to the assumption of exponential life length and independence among daughter particles.
Multiplying by $e^t$ and taking the derivatives we derive the ordinary differential equation  \eqref{ode}. 

Turning to the probability of extinction 
$Q:={\rm P}(Z_\infty=0)$, observe that since
$$Q={\rm E}({\rm P}(Z_\infty=0|Z_t))={\rm E}(Q ^{Z_t})=F_t(Q),\quad t\ge0,$$
equation \eqref{ode} entails that $Q $ is a root of $f(x)=x$. This also gives the smallest non-negative root, because
 \[{\rm P}(Z_t=0)=F_t(0)\nearrow Q,\quad t\to\infty.\]
 \end{proof}
\begin{lemma}
Consider a $(q,r)$-extendable branching process.  
 Then for $t\ge0$, we have $F_t(q)=q$, $F_t(r)=r$, and
\begin{equation}\label{ieq}
t=\int_s^{F_t(s)}{dx\over f(x)-x},\quad s\in[0,q)\cup(q,r).
\end{equation}
\end{lemma}
\begin{proof}
The integral equation \eqref{ieq} follows from the backward Kolmogorov equation \eqref{ode}. The singularity point $x=q$ of the integrand is circumvented as each solution of \eqref{ode} with $x_0\in[0,q)$ is such that $x_t\in[0,q)$ for all $t\ge0$, while each solution of \eqref{ode} with $x_0\in(q,r)$ is such that  $x_t\in(q,r)$  for all $t\ge0$. 
Moreover, \eqref{ode} implies that $F_t(q)=q$ and $F_t(r)=r$, so that for each $t>0$, the probability generating function $F_t(s)$ is $(q,r)$-extendable.
 \end{proof}

Equation \eqref{equ} is obtained from  \eqref{ieq} by extracting principal terms associated with the singularity points $q$ and $r$ of the integrand. We compute these terms with help of the following lemma.

\begin{lemma}\label{lay0}
For a given $(q,r)$-extendable generating function $f$ define $\phi(s)=f^{(3)}(q,r,s)$. Then 
\begin{align*}
\phi(s)&={f(s)-s\over(q-s)(r-s)}, \quad s\in[0,r),\\
 \phi(q)&=f^{(3)}(q,q,r)={\alpha\over r-q},\quad \phi'(q)=f^{(4)}(q,q,q,r)={\alpha\over (r-q)^2}-{f''(q)\over 2(r-q)},
\end{align*}
where $\alpha=1-f'(q)$. Furthermore, if $f'(r)<\infty$, then
\begin{align*}
\phi(r)&=f^{(3)}(q,r,r)={\beta\over r-q},\quad
\phi^{(2)}(q,r)=f^{(4)}(q,q,r,r)={\beta-\alpha\over (r-q)^2},
\end{align*}
where $\beta=f'(r)-1$. Finally, if $f''(r)<\infty$, then
 \[ \phi'(r)=f^{(4)}(q,r,r,r)={f''(r)\over 2(r-q)}-{\beta\over (r-q)^2}.\]
\end{lemma}
\begin{proof}
The first stated equality follows from the definition of $\phi$, $q$, and $r$. Observe that by monotonicity of $\phi$, we have $0<\alpha\le\beta\le\infty$, where  $\alpha=\beta$ holds if and only if $\phi(s)\equiv p_2$ is a constant, that is when the possible numbers of offspring are 0, 1, or 2:
$$f(s)=s+(s-q)(s-r)p_2=qrp_2+(1-qp_2-rp_2)s+p_2s^2.$$ 
This yields one of the statements of Theorem \ref{main} claiming that $\gamma=\alpha/\beta$ belongs to $(0,1]$.

The other stated formulas are obtained using \eqref{lay} and \eqref{deri}. For example, the last statement is valid since
  \[ f^{(4)}(q,r,r,r)={f^{(3)}(r,r,r)-f^{(3)}(q,r,r)\over r-q}={f''(r)\over 2(r-q)}-{f^{(2)}(r,r)-f^{(2)}(q,r)\over (r-q)^2}={f''(r)\over 2(r-q)}-{f'(r)-1\over (r-q)^2}.\]
\end{proof}
%

\begin{proof} {\sc of Theorem \ref{main}}. 
Consider a $(q,r)$-extendable branching process. 
Provided $f'(r)<\infty$, we have
\begin{align*}
{\phi(r)\over (q-x)(r-x)\phi(x)}&={1\over (q-x)(r-x)}+{\phi^{(2)}(r,x)\over (q-x)\phi(x)}
=
{\phi^{(2)}(r,x)(r-q)+\phi(x)\over (r-q)(q-x)\phi(x)}-{1\over (r-q)(r-x)},
\end{align*}
implying
\begin{align*}
{1\over (q-x)(r-x)\phi(x)}&= {\phi^{(2)}(r,x)(r-q)+\phi(x)\over \beta(q-x)\phi(x)}-{1\over \beta(r-x)}.
 \end{align*}
By \eqref{ieq} and Lemma \ref{lay0}, 
\begin{align*}
t&= \int_{s}^{F_t(s)}{dx\over (q-x)(r-x)\phi(x)},
 \end{align*}
and it follows
\begin{align}\label{t1}
t&= {1\over \beta}\ln F_t^{(2)}(r,s)-\int_{s}^{F_t(s)}{\phi^{(2)}(r,x)dx\over \beta\phi(x)}+\int_{s}^{F_t(s)}{\phi(r)dx\over \beta(q-x)\phi(x)},
 \end{align}
 where  $F_t^{(2)}(s_1,s_2)$ is the second order tail generating function for $F_t(s)$.
The last integral equals
\begin{align*}
\int_{s}^{F_t(s)}{dx\over (r-q)(q-x)\phi(x)}&=\int_{s}^{F_t(s)}{dx\over \alpha(q-x)}+\int_{s}^{F_t(s)}{(\phi(q)-\phi(x))dx\over \alpha(q-x)\phi(x)}\\
&=-{ \ln F_t^{(2)}(q,s)\over\alpha}+\int_{s}^{F_t(s)}{\phi^{(2)}(q,x)dx\over \alpha\phi(x)},
 \end{align*}
and we conclude that for $s\in[0,r)$,
\begin{align}\label{equa}
 t={ \ln F_t^{(2)}(r,s)\over\beta}-{\ln F_t^{(2)}(q,s)\over\alpha}+\int_{s}^{F_t(s)}{\phi^{(2)}(q,x)dx\over \alpha\phi(x)}-\int_{s}^{F_t(s)}{\phi^{(2)}(r,x)dx\over \beta\phi(x)},
 \end{align}
which is equivalent to \eqref{equ}, since $\gamma=\alpha/\beta$ and
\[\psi_{q,r}(x)={\phi^{(2)}(q,x)\over \phi(x)}-{\gamma \phi^{(2)}(r,x)\over \phi(x)}.\]
\end{proof}

\section{Tail generating functions for critical branching processes}\label{proco}

In this section we 
assume $f(1)=1$ and $m_1=f'(1)=1$.  Denote $f_{k}(x)=f^{(k)}(1,\ldots,1,x)$  for $k\ge2$,
so that for $x\in[0,1)$,
\begin{equation}\label{fk}
 f_2(x)={1-f(x)\over 1-x},\quad f_3(x)={f(x)-x\over(1-x)^2},\quad f_4(x)={m_2-f_3(x)\over 1-x},\quad f_5(x)={m_3-f_4(x)\over 1-x},
\end{equation}
where $$m_2=f_3(1)={f''(1)\over2},\quad m_3=f_4(1)={f'''(1)\over6}.$$
Notice that in the critical case parameters $m_2$ and $m_3$ are directly related to the centered moments of the reproduction law due to
\begin{align*}
f''(1)&=\sum_{k\ge2}k(k-1)p_k=\sum_{k\ge0}((k-1)^2+k-1)p_k=\sum_{k\ge0}(k-1)^2p_k,\\
f'''(1)&=\sum_{k\ge3}k(k-1)(k-2)p_k=\sum_{k\ge0}((k-1)^2-1)(k-1)p_k=\sum_{k\ge0}(k-1)^3p_k.
\end{align*}

\begin{theorem}\label{main0}
If $f(1)=1$, $m_1=1$, and $f'''(1)<\infty$,
then for $t\ge0$ and  $s\in[0,1)$,
\begin{align}\label{eqc}
t={F_{t}(s)-s\over m_2 (1-s)(1-F_{t}(s))}-{m_3\over m_2^2} \ln {1-F_{t}(s)\over1-s}+
\int_{s}^{F_t(s)}\psi_{1,1}(x)dx,
 \end{align}
 where 
 \begin{align}\label{psihi}
  \psi_{1,1}(x)={f_4^2(x)\over m_2^2f_3(x)}-{f_5(x)\over m_2^2}.
\end{align}
    \end{theorem}

\begin{proof} 
Using \eqref{fk} and  Lemma \ref{ieq}
we derive 
 \begin{align*}
t&= \int_{s}^{F_t(s)}{dx\over (1-x)^2 f_3(x)}=\int_{s}^{F_t(s)}{dx\over (1-x)^2m_2}+\int_{s}^{F_t(s)}{f_4(x)dx\over (1-x)m_2 f_3(x)}.
 \end{align*}
 This gives
 \begin{align*}
m_2t
&={1\over1-F_t(s)}-{1\over1-s}+\int_{s}^{F_t(s)}{ f_4(x)dx\over (1-x) f_3(x)},
 \end{align*}
 which together with 
 \begin{align*}
{ f_4(x)\over (1-x) f_3(x)}-{f_4(x)\over (1-x)m_2}={ f_4^2(x)\over m_2 f_3(x)}
 \end{align*}
yield
 \begin{align*}
m_2^2t
&={m_2(F_{t}(s)-s)\over  (1-s)(1-F_{t}(s))}+\int_{s}^{F_t(s)}{ f_4(x)dx\over 1-x}
+\int_{s}^{F_t(s)}{ f_4^2(x)dx\over  f_3(x)}.
 \end{align*}
Now, to deduce \eqref{eqc}, it remains to use equalities
 \begin{align*}
{ f_4(x)\over 1-x}={m_3\over 1-x}- f_5(x)
,\quad\int_{s}^{F_t(s)}{dx\over 1-x}=\ln {1-F_{t}(s)\over1-s}.
 \end{align*}
 %
\end{proof}

We will show next that the critical case equation \eqref{eqc} is linked to the non-critical case equation \eqref{equa} via a continuity argument, although the components of these two equations look different. For this we need the next observation. 
\begin{lemma} \label{lecon}
Consider a $(q,r)$-extendable branching process with $f'(r)<\infty$, then for every $t\ge0$ and $x\in[0,r)$,
\begin{align}\label{EQ1}
{\psi_{q,r}(x)\over\alpha}&={1\over f^{(3)}(q,q,r)f^{(3)}(q,r,r)}\Big\{
{f^{(4)}(q,q,r,x)f^{(4)}(q,r,r,x)\over f^{(3)}(q,r,x)}
-f^{(5)}(q,q,r,r,x)\Big\},
 \end{align}
and furthermore,
\begin{align}\label{EQ2}
{ \ln F_t^{(2)}(r,s)\over \beta}-{\ln F_t^{(2)}(q,s)\over \alpha}
&
={ 1\over f^{(3)}(r,r,q)}{ F_t^{(3)}(q,r,s)\over  F_t^{(2)}(q,s)}e_{q,r}(t,s)-{f^{(4)}(q,q,r,r)\over f^{(3)}(q,q,r)f^{(3)}(q,r,r)}\ln F_t^{(2)}(q,s),
 \end{align}
where $F_t^{(n)}(s_1,\ldots,s_n)$ is the $n$-th order tail generating function for $F_t(s)$ and 
$$e_{q,r}(t,s)= e\Big\{(r-q){ F_t^{(3)}(q,r,s)\over  F_t^{(2)}(q,s)}\Big\},\quad e\{x\}={\ln(1+x)\over x}.$$
 \end{lemma}

\begin{proof} 
Recall the definition of $\phi$ and its properties obtained in Lemma \ref{lay0}. By a telescopic rearrangement,
\begin{align*}
{\phi^{(2)}(q,x)\over \alpha\phi(x)}-{\phi^{(2)}(r,x)\over\beta \phi(x)}
&={\phi^{(2)}(q,x)\over \alpha\phi(x)}-{\phi^{(2)}(q,x)\over \alpha\phi(r)}+{\phi^{(2)}(q,x)\over \alpha\phi(r)}
-{\phi^{(2)}(r,x)\over\beta \phi(q)}+{\phi^{(2)}(r,x)\over\beta \phi(q)}
-{\phi^{(2)}(r,x)\over\beta \phi(x)}\\
&={(r-q)(r-x)\phi^{(2)}(r,x)\phi^{(2)}(q,x)\over\alpha\beta \phi(x)}+
{(r-q)(\phi^{(2)}(q,x)-\phi^{(2)}(r,x))\over \alpha\beta}\\
&+{(r-q)(x-q)\phi^{(2)}(r,x)\phi^{(2)}(q,x)\over\alpha\beta \phi(x)}={(r-q)^2\over \alpha\beta}\Big\{
{\phi^{(2)}(q,x)\phi^{(2)}(r,x)\over \phi(x)}
-\phi^{(3)}(q,r,x)\Big\},
 \end{align*}
which entails \eqref{EQ1}.
On the other hand, in view of 
 \begin{align*}
 {F_t^{(2)}(r,s)\over F_t^{(2)}(q,s)}=1+{ (r-q)F_t^{(3)}(q,r,s)\over F_t^{(2)}(q,s)},
 \end{align*}
we obtain
\begin{align*}
{ \ln  F_t^{(2)}(r,s)\over \beta}-{\ln F_t^{(2)}(q,s)\over \alpha}
&={ 1\over \beta}\ln\Big\{1+{ (r-q)F_t^{(3)}(q,r,s)\over  F_t^{(2)}(q,s)}\Big\}-{(\beta-\alpha)\ln F_t^{(2)}(q,s)\over \alpha\beta}.
 \end{align*}
It remains to see that the right hand side equals to that of  \eqref{EQ2}.
\end{proof}

Consider a family of $(q_\epsilon,r_\epsilon)$-extendable  branching processes satisfying \eqref{epo} and denote by $F_{t,\epsilon}(s)$ their probability generating functions. If $f'(1)=1$, then the function $F_{t,\epsilon}(s)$ and its limit $F_t(s)$ satisfy equations \eqref{equa} and \eqref{eqc} respectively. 
Applying Lemma \ref{lecon} we see that  there is a term by term agreement between \eqref{equa} and \eqref{eqc}.
Indeed, by \eqref{EQ1} we can write $\alpha_\epsilon^{-1}\psi_{q_\epsilon,r_\epsilon}(x)\to\psi_{1,1}(x)$.
On the other hand,  $e_{q_\epsilon,r_\epsilon}(t,s)\to1$ as $r_\epsilon-q_\epsilon\to0$, so that \eqref{EQ2} eventually implies 
\begin{align*}
{ \ln F_{t,\epsilon}^{(2)}(r_\epsilon,s)\over \beta_\epsilon}-{\ln  F_{t,\epsilon}^{(2)}(q_\epsilon,s)\over \alpha_\epsilon}
&\to { 1\over m_2}\,{ F_t^{(3)}(1,1,s)\over  F_t^{(2)}(1,s)}-{m_3\over m_2^2}\,\ln F_t^{(2)}(1,s).
\end{align*}

\section{Modified linear-fractional reproduction law}\label{Ex2}
\begin{definition}\label{mlf}
A (possibly deffective) probability distribution will be called modified linear-fractional, if its generating function has the form
 \begin{equation}\label{moj}
f(s)=p_0+p_1s+(1-p_0-p_1-p_\Delta) s^2(1-p)(1-ps)^{-1},\quad s\in[0,p^{-1}),
\end{equation}
for some combination of fours parameters $(p_0,p_1,p_\Delta,p)$ satisfying
$$p_0\in[0,1),\quad p_1\in[0,1),\quad p_\Delta\in[0,1),\quad p_0+p_1+p_\Delta<1,\quad p\in[0,1).$$
\end{definition}

A random variable $X$ having a  modified linear-fractional distribution
is characterised by the following shifted geometric property 
\[{\rm P}(X=2+k|2\le X<\infty)=(1-p)p^k,\quad k\ge0.\]
Definition \ref{mlf} is a generalisation of the well-known linear-fractional (or zero-modified geometric) distribution with
\[f(s)=p_0+(1-p_0)s(1-p)(1-ps)^{-1}.
\]
Indeed, putting $p_1=(1-p_0-p_\Delta)(1-p)$ into \eqref{moj} we arrive at a possibly deffective linear-fractional generating function 
 \begin{equation}\label{mol}
f(s)=p_0+(1-p_0-p_\Delta) s(1-p)(1-ps)^{-1},
\end{equation}
with $p_0, p_\Delta, p\in[0,1)$ and $p_0+p_\Delta<1$.

\begin{lemma}\label{necr}
Consider a modified linear-fractional $f$ given by Definition \ref{mlf}. If 
\[p_\Delta=0, \quad p_0\in(0,1),\quad p_1=1-p_0(2-p)\in[0,1),\quad p\in[0,1),\]
then
\[
f(1)=f'(1)=1,\quad m_{2}={p_0 \over  1-p},\quad m_{3}={p_0p \over  (1-p)^{2}},
\]
and
\begin{equation}\label{ecr}
  f(s)=s+p_0(1-s)^2(1-ps)^{-1},\quad p_0\in(0,1),\quad p\in[0,1).
\end{equation}
In all other cases $f$ is a $(q,r)$-extendable probability generating function such that
\begin{equation}\label{mojo}
f(s)
=s+{\alpha(q-s)(r-s)\over r-q\gamma-(1-\gamma)s},\quad s\in|0,r),
\end{equation}
where besides the usual conditions on $(q,r,\alpha,\gamma)$:
\begin{equation}\label{un}
0\le q\le1\le r<\infty,\quad q<r,\quad \alpha\in(0,1),\quad \gamma\in(0,1],
\end{equation}
the following extra restriction holds
\begin{equation}\label{und}
 \alpha\le 1-\gamma+{ (r- q)^2\gamma\over r^2-\gamma q^2}.
\end{equation}
\end{lemma}
\begin{proof} Clearly, in  the modified linear-fractional distribution case
\[m_1=p_1+(1-p_0-p_1-p_\Delta){2-p\over 1-p},\]
so that given $f(1)=1$ we get
\[m_1
=1+{1-p_0(2-p)-p_1\over 1-p}.\]
Having this in mind, the critical case formula \eqref{ecr} is easily checked. Observe also that in the critical case
 \[f_{2+k}(s)=f^{(2+k)}(1,\ldots,1,s)={p_0p^k \over  (1-p)^{k}(1-ps)},\quad k\ge0.
\]

In the non-critical case, taking into account that $q$ and $r$ are roots of equation $f(x)=x$, relation \eqref{moj} can be rewritten as
\[
f(s)=s+{(q-s)(r-s)c\over 1-ps},
\]
with some $c\in(0,\infty)$. This means $\phi(s)={c\over 1-ps}$, so that in view of $\phi(q)={\alpha\over 1-pq}$ and $\phi(r)={\beta\over 1-pr}$ we obtain \eqref{mojo}. Comparing  \eqref{mojo} with \eqref{moj}, we find
\begin{equation}\label{ppp}
 p_0={ \alpha  qr\over  r-\gamma q},
\quad p_1=1-{\alpha(r^2-\gamma q^2) \over  (r-\gamma q)^2},
\quad p_\Delta={ \alpha(r-1)(1-q) \over   r-1+\gamma (1-q)},\quad p={1 -\gamma \over  r-\gamma q}.
\end{equation}
These relations imply that  the conditions $p_0\ge0$, $p_1<1$, $p_\Delta\ge0$, $0\le p< 1$ are always satisfied. Restriction \eqref{und} stems from $p_1\ge0$. No further restrictions are needed, since  
\begin{align*}
 1-p_0-p_1-p_\Delta&={\alpha(r^2-\gamma q^2) \over  (r-\gamma q)^2}-{ \alpha qr\over  r-\gamma q}-{ \alpha(r-1)(1-q) \over   r-1+\gamma (1-q)}\\
&={\alpha\big[(r^2(1-q)+\gamma q^2(r-1))(r-1+\gamma (1-q)) -(r-1)(1-q)(r-\gamma q)^2\big]\over (r-\gamma q)^2( r-1+\gamma (1-q))}\\
&={\alpha\gamma(r-q)^2\over (r-\gamma q)^2( r-1+\gamma (1-q))}
\end{align*}
implies $p_0+p_1+p_\Delta<1$.
\end{proof}
\begin{proposition}
For the functions defined by \eqref{psih} and  \eqref{psihi}, condition $\psi_{q,r}(x)\equiv 0$ holds if and only if $f$ has the form \eqref{moj}. Given  \eqref{moj},
the probability generating function $F_t(s)$ of the corresponding branching process satisfies in the non-critical case
\begin{equation}\label{lfu}
{F_t(s)-q\over s-q}= e^{- \alpha t}\Big\{{r-F_t(s)\over r-s}\Big\}^\gamma ,
\end{equation}
and in the critical case,  with $f$ given by \eqref{ecr},
\begin{equation}\label{lfc}
p_0 t=(1-p){F_{t}(s)-s\over (1-s)(1-F_{t}(s))}-p  \ln {1-F_{t}(s)\over1-s}.
\end{equation}
\end{proposition}
\begin{proof}
To prove the stated criterium, observe that in terms of  Lemma \ref{lay0}, condition  $\psi_{q,r}(x)\equiv 0$ is equivalent to
\begin{equation}\label{fi}
 \beta \phi^{(2)}(q,s)= \alpha \phi^{(2)}(r,s),\quad s\in|0,r).
\end{equation}
Using relations from Lemma \ref{lay0} we see that the last relation is equivalent to   \eqref{mojo}, which in turn is equivalent to \eqref{moj} by Lemma \ref{necr}.
Equations \eqref{lfu} and \eqref{lfc} directly follow from Theorems \ref{main} and \ref{main0}. 
\end{proof}

\begin{remark}
 According to \eqref{ppp} with the special choice of $(q,r)=(0,1)$, the modified linear-fractional probability generating function \eqref{moj} becomes
 \[h(s)=(1-\alpha)s+{\alpha\gamma s^2\over 1-(1-\gamma) s}.
 \]
 On the other hand, for any $f$ given by  \eqref{moj}, we have
 \[f^{(2)}(q,q+(r-q)s)
 ={\gamma s+(1-\alpha)(1-s)\over 1-(1-\gamma) s}={h(s)\over s}.\]
 This implies a representation
 \[f(s)=q+(r-q)h\Big({s-q\over r-q}\Big)\]
that can be interpreted in the following way. We can treat  the pair of fixed points $(q,r)$ as scaling parameters, and the pair $(\alpha,\gamma)$ as shape parameters for the family of modified linear-fractional distributions. Recall that parametrisation $(q,r,\alpha,\gamma)$ is subject to restrictions \eqref{un} and \eqref{und}. 

Notice also, that a linear-fractional generating function \eqref{mol} is fully defined by a triplet $(q,r,\gamma)$ which corresponds to $(q,r,\alpha,\gamma)$ with
 $\alpha=1-\gamma$.  In this case restriction \eqref{und} is fulfilled automatically.
\end{remark}

\section{Reproduction with trifurcations}\label{Ex}
Putting $p_1=p=0$ into \eqref{moj}, we get  a (possibly defective) binary splitting reproduction law 
$$f(s)=p_0+p_2s^2, \qquad p_0+p_2\le 1.$$
The corresponding branching process is the linear birth-death process with killing studied in \cite{KT}.
In this case equation \eqref{lfu} holds with $\gamma=1$ and
\[\alpha=\sqrt{1-4p_0p_2},\qquad q={1-\alpha\over 2p_2},\qquad r={1+\alpha\over 2p_2}.\]
It brings  the well-known explicit linear-fractional solution for the non-critical case
\[F_t(s)={(r-s)q+(s-q)re^{-\alpha t}\over r-s+(s-q)e^{-\alpha t}}.\]
In the critical case, $p_0=p_2={1\over 2}$, equation  \eqref{lfc} becomes
\[1-{1-F_t(s)\over 1-s}={t\over 2}\,(1-F_t(s)),
\]
which yields the  linear-fractional formula for the critical birth-death process
\[F_t(s)=1-{1-s\over 1+{t\over 2}(1-s)}.
\]
Further examples of explicit formulas for $F_t(s)$, going beyond the linear-fractional case, are presented in \cite{SL}.

A less trivial example arises when {\it trifurcations} are also allowed. Consider a three-parameter family
\[f(s)=p_0+p_2s^2+p_3s^3,\quad p_3>0,\quad p_0+p_2+p_3\le1.\]
Denote by $(q,r,x_3)$ the roots of the third order algebraic equation $f(x)=x$: two non-negative roots $q\le r$ and a negative solution $x_3$. Then we can write
\[f(s)-s=p_3(s-q)(s-r)(s-x_3)=(s-q)(s-r)(p_3s+w),\]
where $w=-p_3x_3\in(0,\infty)$. From
\[f(1)=1-(1-q)(r-1)(p_3+w),\]
it is clear that $q\le1\le r$. 
Since $f'(0)=p_1=0$, and 
\[f'(s)=1+(s-r)(p_3s+w)+(s-q)(p_3s+w)+p_3(s-q)(s-r),\]
we conclude that
$w={1+p_3qr\over q+r}.$

\begin{proposition}\label{tro}
Consider a branching process with the reproduction law
\[f(s)=s+(s-q)(s-r)(p_3s+w),\quad w={1+p_3qr\over q+r}.\]
If $q<r$, then
\begin{equation}\label{troe}
 {F_t(s)-q\over s-q}= e^{- \alpha t}\Big\{{r-F_t(s)\over r-s}\Big\}^\gamma  \Big({p_3F_t(s)+w\over p_3s+w}\Big)^{1-\gamma},\quad \gamma={p_3q+w\over p_3r+w}.
\end{equation}
If $q=r=1$,  then
\[
t={2\over 1+3p_3}{F_{t}(s)-s\over (1-s)(1-F_{t}(s))}-{4p_3\over (1+3p_3)^2} \ln {1-F_{t}(s)\over1-s}+{4p_3\over (1+3p_3)^2}\ln{1+p_3+2p_3F_t(s)\over 1+p_3+2p_3s}.
 \]
%
\end{proposition}
\begin{proof}
After computing 
$ f^{(3)}(q,r,s)=p_3s+w$, we find
\begin{align*}
 \alpha&=(r-q)(p_3q+w),\quad  \beta=(r-q)(p_3r+w),
 \quad
f^{(4)}(q,q,r,s)\equiv f^{(4)}(q,r,r,s)\equiv p_3,
  \end{align*}
so that given $q<r$, the function defined by \eqref{psih} is computed explicitly
 \[\psi_{q,r}(x)={(r-q)p_3^2 \over (p_3r+w)(p_3x+w)},\quad \int_{s_1}^{s_2}\psi_{q,r}(x)dx=(1-\gamma)\ln {p_3s_2+w \over p_3s_1+w}.\]
As a result, equation \eqref{equ} simplifies and takes the form  stated by the lemma.

In the critical case when $q=r=1$, we  get
\[f(s)=s+{1\over 2}(1-s)^2(2p_3s+1+p_3),\quad f_3(s)=p_3s+{1+p_3\over2},\quad f_4(s)=p_3,\quad f_5(s)=0,\]
implying $m_2={1+3p_3\over2}$ and $m_3=p_3$. It follows that
\[\psi_{1,1}(x)={2p_3^2 \over m_2^2( 1+p_3+2p_3x)},\quad \int_{s_1}^{s_2}\psi_{1,1}(x)dx={4p_3\over (1+3p_3)^2}\ln{1+p_3+2p_3s_2\over 1+p_3+2p_3s_1},\]
and \eqref{eqc} implies the second stated equation.
\end{proof}

\noindent{\bf Remark}. If $f(s)=s^3$, then  
$q=0$, $r=p_3=w=1$, and $\gamma=1/2,$
so that   equation \eqref{troe} takes the form
\[ {e^{2 t}F_t^2(s)\over s^2}= {1-F_t^2(s)\over 1-s^2},\]
which can be solved explicitly. 
This is a particular case of the Harris-Yule process characterised by $f(s)=s^{k+1}$ for some $k\ge1$. In this case an explicit expression is available:
\[F_t(s)=\Big(e^{kt}s^{-k}-e^{kt}+1\Big)^{-1/k}.\]

\section{Tail generating functions and $xlogx$-conditions}\label{Sta}

 In this section we establish Theorem \ref{tail}, which  presents a criterium for a generalised $x\log x$ condition in terms of the tail generating functions. Using Theorem \ref{tail} we prove Propositions \ref{pipi4} and \ref{pipi1} addressing condition 
$\int_0^r|\psi_{q,r}(x)|dx<\infty$
 for the functions \eqref{psih} and \eqref{psihi}.

We start by showing that the earlier announced relation \eqref{deri} holds. Indeed, turning to Definition \ref{def}, we find
\begin{align*}
v^{(k+n)}(s_1,\ldots,s_{k-1},s,\ldots,s)
&=\sum_{i_1\ge0,\ldots,i_n\ge0}v_{j_{k+n}}s_1^{i_1}\ldots s_{k-1}^{i_{k-1}} s^{i_k+\ldots+i_{k+n}}\\
&=\sum_{i_1\ge0,\ldots,i_k\ge0}{n+i_k\choose n}v_{n+j_k}s_1^{i_1}\ldots s_{k-1}^{i_{k-1}} s^{i_k}={1\over n!}{d^n\over ds^n}v^{(k)}(s_{1},\ldots,s_{k-1}, s),
\end{align*}
where $ j_k=i_1+\ldots+i_k+k-1$. In particular,
\begin{equation}\label{vev}
v^{(n+2)}(a,\ldots,a,s)=\sum_{i\ge0}s^{i}\sum_{j\ge n}a^{j-n}{j\choose n}v_{i+j+1}, \quad n\ge0.
\end{equation}
\begin{theorem}\label{tail}
Let $f(s)=\sum_{k=0}^\infty s^kp_k$ be a (possibly deffective) probability generating function, $a>0$, and $n\ge0$ be a non-negative integer. Then, the  moment condition
\begin{equation}\label{vlv}
 \sum_{k=2}^\infty p_ka^kk^{n}\ln k<\infty
\end{equation}
is equivalent to
$$\int_0^a f^{(n+2)}(a,\ldots,a,x)dx<\infty.$$
\end{theorem}

\begin{proof}
Applying \eqref{vev} we find
\begin{align*}
\int_0^a f^{(n+2)}(a,\ldots,a,x)dx
&=\sum_{j=0}^{\infty} p_{j+n}\sum_{i=0}^j{n+i\choose i}a^i\int_0^ax^{j-i}dx\\
&={1\over n!}\sum_{j=0}^{\infty} p_{j+n}a^{j+1}\sum_{l=0}^j{1\over l+1}(n+j-l)\cdots(1+j-l)
\end{align*}
for all $n\ge1$.
Since
\[ \sum_{l=0}^j{1\over l+1}\prod_{i=1}^{n}(i+j-l)=j^n \sum_{l=0}^j{1\over l+1}\prod_{i=1}^{n}(1+(i-l)j^{-1})\sim j^{n}\ln j,\quad j\to\infty,\]
the statement follows.
\end{proof}

\begin{proposition}\label{pipi4}
Consider a $(q,r)$-extendable probability generating function $f$  with $f'(r)<\infty$ and the corresponding function \eqref{psih}. We have
$Ê\int_0^{r}|\psi_{q,r}(x)|dx<\infty,$
if and only if
\begin{equation}\label{rx}
 \sum_{k=2}^\infty p_kr^kk\ln k<\infty.
\end{equation}
If  condition \eqref{rx} does not hold, then the function
$\mathcal L_{q,r}(x)=\exp\Big\{\int_0^{r-x}\psi_{q,r}(s)ds\Big\}$
slowly varies as $x\to 0$ and $\mathcal L_{q,r}(x)\to0$.
\end{proposition}

\begin{proposition}\label{pipi1}
Let $f(1)=1$, $f'(1)=1$, and $f'''(1)<\infty$ and consider the function \eqref{psihi}. We have
$\int_0^1|\psi_{1,1}(x)|dx<\infty$
if and only if
\begin{equation}\label{xxx}
 \sum_{k=2}^\infty p_kk^3\ln k<\infty.
\end{equation}
If \eqref{xxx} does not hold, then the function
$\mathcal L_{1,1}(x)=\exp\Big\{\int_0^{1-x}\psi_{1,1}(s)ds\Big\}$
slowly varies as $x\to 0$ and $\mathcal L_{1,1}(x)\to0$.
\end{proposition}

\begin{proof}
Propositions \ref{pipi4} and \ref{pipi1} have similar proofs. Here we prove only  Proposition \ref{pipi4}. 
Applying  Lemma \ref{lay0}, 
 we see that 
 \[\phi(s)\in[\phi(0),\phi(r)]\subset(0,\infty),\quad s\in[0,r].\]
Thus, in view of $\phi^{(2)}(q,r)<\infty$, we have 
$$Êc_{q,r}:=\int_0^{r}{\phi^{(2)}(q,x)dx\over \phi(x)}<\infty,$$
and it suffices to verify that
\begin{equation}\label{11q}
\int_0^r\phi^{(2)}(r,x)dx<\infty
\end{equation}
if and only if \eqref{rx} holds (the integral in \eqref{11q} may be infinite because
$\phi^{(2)}(r,r)$
is allowed to be infinite). Indeed, since
\[\phi^{(2)}(s_1,s_2)=f^{(4)}(q,r,s_1,s_2)={f^{(3)}(r,s_1,s_2)-f^{(3)}(q,r,s_1)\over s_2-q},\]
we have
\[\int_{(r+q)/2}^r\phi^{(2)}(r,x)dx=\int_{(r+q)/2}^r{f^{(3)}(r,r,x)-f^{(3)}(q,r,r)\over x-q}dx= \int_{(r+q)/2}^r{f^{(3)}(r,r,x)\over x-q}dx-{\beta\over r-q}\ln{2},\]
implying that \eqref{11q} is equivalent to
$\int_0^rf^{(3)}(r,r,x)dx<\infty,$
which in turn  is equivalent to \eqref{rx} by Theorem \ref{tail}. 
To finish the proof of  Proposition \ref{pipi4}, notice that slow variation of $\mathcal L_{q,r}(x)$ follows from the representation
 \[\mathcal L_{q,r}(x)\sim \exp\Big\{c_{q,r}-\gamma\int_x^r{\eta(s) ds\over  s}\Big\},
 \]
 where  
  \[\eta (r-s)={(r-s)\phi^{(2)}(r,s)\over \phi(s)}={\phi(r)-\phi(s)\over \phi(s)},\]
  so that $\eta (x)\to0$ as $x\to0$, see \cite{Bi}.
\end{proof}

\noindent{\bf Examples.} A possibility for $f'(r)<\infty$ and $f''(r)=\infty$  is illustrated by the next example borrowed from \cite{SL}. For a given set of four parameters $(q,r,a,\theta)$ satisfying $0\le q\le1<r<\infty$, $a\in(0,1)$, $\theta\in(0,1)$, the function
\[f(s)=r-\{a(r-s)^{-\theta}+(1-a)(r-q)^{-\theta}\}^{-1/\theta}\]
is a $(q,r)$-extendable probability generating function. For this example, we have  $f'(q)=a$, $f'(r)=a^{-1/\theta}$,  $f''(r)=\infty$, and 
\begin{align*}
 f^{(2)}(r,s)&=\{a+(1-a)(r-s)^{\theta}(r-q)^{-\theta}\}^{-1/\theta},\\
f^{(3)}(q,r,s)&={\{a+(1-a)(r-s)^{\theta}(r-q)^{-\theta}\}^{-1/\theta}-1\over s-q},\\
f^{(4)}(q,r,s)&={a^{-1/\theta}(s-q)-(r-q)\{a+(1-a)(r-s)^{\theta}(r-q)^{-\theta}\}^{-1/\theta}\over (r-q)(s-q)(r-s)}+{1\over (r-q)(s-q)}.
\end{align*}
Since 
\[f^{(4)}(q,r,s)\sim{a^{-1/\theta}(1-a)\over \theta (r-q)^{1+\theta}(r-s)^{1-\theta}},\quad s\to r,\]
we conclude that in this case $Ê\int_0^{r}|\psi_{q,r}(x)|dx<\infty$.

A related example from  \cite{SL} introduces the case $f'(r)=\infty$, which is  not studied here: if $a\in(0,1)$ and $q\in[0,1]$, then
\[f(s)=r-(r-q)^{1-a}(r-s)^a\]
is a $(q,r)$-extendable probability generating function such that $f^{(2)}(r,s)=({r-q\over r-s})^{1-a}$.

\section{Yaglom-type limit theorem}\label{Ssub}

With $f(1)<1$, a realisation of the  branching process has two possible fates:  either to be absorbed at the state $0$ at a random time $T_0$, or to be absorbed at the graveyard state $\Delta$ at a random time $T_1$. Indeed, by \eqref{equ}, we have 
\begin{align}\label{equr}
{F_t(1)-q\over 1-q}= e^{- \alpha t}\Big\{{r-F_t(1)\over r-1}\Big\}^\gamma \exp\Big\{-\int^{1}_{F_t(1)}\psi_{q,r}(x)dx\Big\}.
 \end{align}
 Thus, provided $q<1<r$, we get 
 $${\rm P}(Z_t= \Delta)=1-F_t(1)\to 1-q,\quad t\to\infty.$$
In the defective case, for the overall absorption time
$$T:=\min(T_0,T_1)=T_0\cdot1_{\{T_0<\infty,T_1=\infty\}}+T_1\cdot1_{\{T_0=\infty,T_1<\infty\}}+\infty\cdot1_{\{T_0=\infty,T_1=\infty\}},$$
we obtain ${\rm P}(T=\infty)=0$ and 
\[{\rm P}(T>t)= {\rm P}(t<T_0<\infty)+ {\rm P}(t<T_1<\infty)=F_t(1)-F_t(0),\]
since 
\begin{align*}
 {\rm P}(t<T_0<\infty)&={\rm P}(Z_t\ne0, Z_\infty=0)={\rm P}(Z_\infty=0)-{\rm P}(Z_t=0)=q-F_t(0),\\
{\rm P}(t<T_1<\infty)&={\rm P}(Z_\infty=\Delta)-{\rm P}(Z_t=\Delta)=1-q-(1-F_t(1))=F_t(1)-q.
\end{align*}
We will establish an asymptotic formula for ${\rm P}(T>t)$ as $t\to \infty$, using the following result for $q-F_t(s)$, which is also valid for $f(1)=1$. 

\begin{lemma}\label{bcr1}  In the $(q,r)$-extendable case, for a given $s\in[0,r)$, we have
\begin{align*}
F_t(s)
&=q+K(s)e^{-\alpha t}+{f''(q)\over 2\alpha}K^2(s)e^{-2\alpha t}+o(e^{-2\alpha t}),\quad t\to\infty,
\end{align*}
where 
\[K(s)=(s-q)(r-q)^\gamma (r-s)^{-\gamma} \exp\Big\{\int_s^q\psi_{q,r}(x)dx\Big\}.
\]
\end{lemma}

\begin{proof}
For $s=q$ the assertion is trivial.  By Theorem \ref{main}, for $s\in[0,r)$ and $s\neq q$, we have
 \begin{align*}
F_t(s)-q
&=(s-q)e^{-\alpha t}\Big({r-F_t(s)\over r-s}\Big)^\gamma
\exp\Big\{\int_{s}^{F_t(s)}\psi_{q,r}(x)dx\Big\}\\
&=K(s)e^{-\alpha t}\Big(1-{F_t(s)-q\over r-q}\Big)^\gamma
\exp\Big\{\int_{q}^{F_t(s)}\psi_{q,r}(x)dx\Big\}.
 \end{align*}
 It remains to observe that as $t\to\infty$,
 \begin{align*}
1-\Big(1-{F_t(s)-q\over r-q}\Big)^\gamma&=(F_t(s)-q)\Big({\gamma\over r-q}+o(1)\Big)={\gamma\over r-q}K(s)e^{-\alpha t}+o(e^{-\alpha t}),\\
\exp\Big\{\int_{q}^{F_t(s)}\psi_{q,r}(x)dx\Big\}-1&=(F_t(s)-q)\Big(\psi_{q,r}(q)+o(1)\Big)\\
&={\phi^{(2)}(q,q)-\gamma\phi^{(2)}(q,r)\over \phi(q)}\cdot K(s)e^{-\alpha t}+o(e^{-\alpha t}),
 \end{align*}
 and that
\begin{align*}
{\gamma\over r-q}+{\gamma\phi^{(2)}(q,r)-\phi^{(2)}(q,q)\over \phi(q)}={\gamma\over r-q}+{1-\gamma\over r-q}-{1\over r-q}+{f''(q)\over 2\alpha}={f''(q)\over 2\alpha}.
 \end{align*}
\end{proof}

In the supercritical case, when $0\le q<1=r$, with probability $1-q$ the branching process grows exponentially forever without being absorbed at zero or $\Delta$, so that ${\rm P}(T=\infty)=1-q$. This case  is excluded in the next asymptotic result.
\begin{theorem}\label{cde}
In the $(q,r)$-extendable case with $r>1$ and $f'(r)<\infty$, we have
\[{\rm P}(T>t)= (K(1)-K(0))e^{-\alpha t}+{f''(q)\over 2\alpha}(K^2(1)-K^2(0))e^{-2\alpha t}+o(e^{-2\alpha t}),\]
as $t\to\infty$, and moreover,
\begin{align*}
{\rm P}(Z_t=k|T>t)&\to \pi_k,\quad k\ge1,\quad \sum_{k=1}^\infty \pi_ks^k={K(s)-K(0)\over K(1)-K(0)}.
\end{align*}
If $f$ is modified linear-fractional, then 
\[K(s)=(s-q)(r-q)^\gamma (r-s)^{-\gamma},
\]
and 
$ \pi_k\sim{c k^{1-\gamma} r^{-k}}$ as $k\to\infty$ for some positive constant $c$.
\end{theorem}

\begin{proof} Applying Lemma \ref{bcr1} we arrive at the first statement.
 The stated conditional weak convergence is also a consequence of Lemma \ref{bcr1} 
\begin{align*}
{\rm E}(s^{Z_t}|T>t)={{\rm E}(s^{Z_t})-{\rm E}(s^{Z_t};T\le t)\over {\rm P}(T>t)}={F_t(s)-F_t(0)\over F_t(1)-F_t(0)}\to{K(s)-K(0)\over K(1)-K(0)}.
\end{align*}
Note that 
 \[K(rs)=(rs-q)(r-q)^\gamma r^{-\gamma}(1-s)^{-\gamma} \exp\Big\{\int_{rs}^q\psi_{q,r}(x)dx\Big\}\sim (r-q)^{1+\gamma} r^{-\gamma}(1-s)^{-\gamma}\mathcal L(1-s),
\]
as $s\to 1$, where
 \[\mathcal L(1-s)=\exp\Big\{-\int_q^{rs}\psi_{q,r}(x)dx\Big\}
\]
is a slowly varying function according to Proposition \ref{pipi4}. Therefore, by the Tauberian theorem, we have
 \begin{align*}
\sum_{k=1}^n \pi_kr^k\sim n^\gamma l_n,\quad n\to\infty,
\end{align*}
where $l_n$ is a positive  slowly varying sequence.

Turning to a modified linear-fractional $f$, we use
 \begin{align*}
 (1-x)^{-\gamma}=\sum_{k=0}^\infty \gamma_kx^k,\quad \gamma_k=\prod_{i=1}^k\Big(1-{1-\gamma\over i}\Big),
 \end{align*}
to find
 \begin{align*}
(s-q)
(1-s/r)^{-\gamma}
=-q+\sum_{k=1}^\infty \gamma_{k-1}\Big(r-q+{(1-\gamma)q\over k}\Big){s^k\over r^k}.
 \end{align*}
Thus
\begin{align*}
{K(s)-K(0)\over K(1)-K(0)}=\sum_{k=1}^\infty c_k{\gamma_{k-1} r^{-k}}s^k,\quad c_k={1\over  q+(1-q) (1-r^{-1})^{-\gamma}}\Big(r-q+{(1-\gamma)q\over k} \Big),
\end{align*}
and it remains to see that
$ c_k\gamma_{k-1}\sim (r-q)k^{1-\gamma}/c_\gamma $ as $k\to\infty$, where $c_\gamma=\Gamma(\gamma)(q+(1-q) (1-r^{-1})^{-\gamma})$.
\end{proof}

\begin{remark}
 It is interesting to see how two fixed points $q$ and $r$ regulate different aspects of the non-absorption behavior. While the rate of decay of the non-absorption probability is controlled by $\alpha=1-f'(q)$, the conditional distribution tails are ruled by the value of $r$, in that $\pi_k$ is of order $r^{-k}$.
\end{remark}

\

\noindent {\bf Example.} In the framework of  Proposition \ref{tro} we find
\[
K(s)=(s-q)(r-q)^\gamma (r-s)^{-\gamma}\Big({p_3q+w\over p_3s+w}\Big)^{1-\gamma}.
\]
Since 
\[
K(rs)\sim(r-q)^{\gamma+1}r^\gamma \gamma^{1-\gamma} (1-s)^{-\gamma},\quad s\to1,
\]
we see that in this case $\sum_{k=1}^n \pi_kr^k\sim cn^\gamma$ as $n\to\infty$ for some positive $c$.

 \section{Limit theorems for the termination time}\label{ex3}
 In this section we consider a family of $(q_\epsilon,r_\epsilon)$-extendable branching processes satisfying \eqref{epo}, where, without loss of generality, it is assumed that 
$\epsilon=1-f_\epsilon(1)$. We obtain weak convergence results as  $\epsilon\to0$ for the termination time $T_{1,\epsilon}$ conditioned on $T_{1,\epsilon}<\infty$.
 
\begin{lemma}\label{near}
 In the nearly subcritical case, when $0<m_1<1$, we have
\[1-q_\epsilon\sim{\epsilon\over 1-m_1},\quad r_\epsilon\to r\in(1,\infty),\quad \alpha_\epsilon\to 1-m_1\in(0,1),\quad \beta_\epsilon\to \beta\in(0,\infty).\]
In the nearly supercritical case, when $m_1>1$, we have
\[r_\epsilon-1\sim{\epsilon\over m_1-1},\quad q_\epsilon\to q\in[0,1),\quad \alpha_\epsilon\to \alpha\in(0,1),\quad \beta_\epsilon\to m_1-1\in(0,\infty).\]
In the nearly critical case, when  $m_1=1$, we have 
$$q_\epsilon\to 1,\quad r_\epsilon\to 1,\quad \alpha_\epsilon\to 0,\quad \beta_\epsilon\to 0,\quad \gamma_\epsilon\to 1,$$ 
and  if it is given that
\begin{equation}\label{d}
 d_\epsilon:={1-q_\epsilon\over r_\epsilon-1}\to d\in[0,\infty],
\end{equation}
then
\[{1-q_\epsilon\over\sqrt{\epsilon}}\to\sqrt{d\over m_2 },\quad {r_\epsilon-1\over\sqrt{\epsilon}}\to{1\over\sqrt{m_2 d}}.\]
\end{lemma}
\begin{proof} The first two assertions follow from the equalities
\[
\epsilon=\phi_\epsilon(1)(1-q_\epsilon)(r_\epsilon-1),\quad \phi_\epsilon(q_\epsilon)={\alpha_\epsilon\over r_\epsilon-q_\epsilon},\quad \phi_\epsilon(r_\epsilon)={\beta_\epsilon\over r_\epsilon-q_\epsilon},\]
obtained from Lemma \ref{lay0}. 
In the  nearly critical case, since
\[\phi_\epsilon(q_\epsilon)\to m_2,\quad \phi_\epsilon(r_\epsilon)\to m_2,\quad \phi^{(2)}_\epsilon(q_\epsilon,r_\epsilon)\to m_3,\]
we get
\[
m_2(1-q_\epsilon)(r_\epsilon-1)\sim\epsilon,\quad {\alpha_\epsilon\over r_\epsilon-q_\epsilon}\to m_2,\quad {\beta_\epsilon\over r_\epsilon-q_\epsilon}\to m_2
.\]
Using these relations it is easy to verify the statements for the nearly critical case.
\end{proof}

\begin{theorem}\label{rob}
 Consider a family of $(q_\epsilon,r_\epsilon)$-extendable branching processes satisfying \eqref{epo} and let $\epsilon\to0$.
 
(i) If $m_1<1$, then 
for any fixed $t\ge0$,
\[
{\rm P}\Big(T_{1,\epsilon}\le t\,\Big|\,T_{1,\epsilon}<\infty\Big)
\to1-e^{(1-m_1) t}.
\]

(ii) If $m_1>1$, then 
 for any fixed $u\in(-\infty,\infty)$,
\[
 {\rm P}\Big(T_{1,\epsilon}\le{\ln (r_\epsilon-1)^{-1}\over \beta_\epsilon}+{\ln (1-q)+u\over m_1-1}\,\Big|\,T_{1,\epsilon}<\infty\Big)\to\Phi(u),
\]
 and 
 the limit distribution function  satisfies 
\begin{align}\label{u}
u&= \ln \Phi(u)+\int^{1}_{1-(1-q)\Phi(u)}\psi(x)dx,
 \end{align}
where the function
\[
\psi(x)={\phi^{(2)}(1,x)(1-q)\over \phi(x)}+{m_1-1\over(1-q)(x-q)\phi(x)} \]
takes positive values over $x\in(q,1]$.
 \end{theorem}

\begin{proof} 
Put
 \[V_\epsilon(t)={\rm P}\Big(T_{1,\epsilon}\le t\,\Big|\,T_{1,\epsilon}<\infty\Big)
\]
and observe that
\[
V_\epsilon(t)={1-F_{t,\epsilon}(1)\over 1-q_\epsilon}=1-{F_{t,\epsilon}(1)-q_\epsilon\over 1-q_\epsilon}.
\]

(i) Referring to \eqref{equr} we can write an equation for $V_\epsilon(t)$
\begin{equation}\label{eto}
 1-V_\epsilon(t)=e^{-\alpha_\epsilon t} \Big[1+{1-q_\epsilon\over r_\epsilon-1} V_\epsilon(t)\Big]^{\gamma_\epsilon}
 \exp\Big\{-\int_{1-(1-q_\epsilon)V_\epsilon(t)}^1\psi_{q_\epsilon,r_\epsilon}(x)dx\Big\}.
\end{equation}
This and the first part of Lemma \ref{near} imply the assertion in the nearly subcritical case. 


(ii) In the nearly supercritical case applying \eqref{t1} with $s=1$ we get
\begin{align*}
\beta_\epsilon t+\ln (r_\epsilon-1)&= \ln (r_\epsilon-F_{t,\epsilon}(1))+\int^{1}_{F_{t,\epsilon}(1)}{\phi^{(2)}_\epsilon(r_\epsilon,x)dx\over \phi_\epsilon(x)}+\int^{1}_{F_{t,\epsilon}(1)}{\phi(r_\epsilon)dx\over (x-q_\epsilon)\phi_\epsilon(x)}.
 \end{align*}
It follows that the time scaled distribution function $\Phi_\epsilon(u)=V_\epsilon(t_\epsilon(u))$, where
\begin{align*}
 t_\epsilon(u)={\ln (r_\epsilon-1)^{-1}+\ln (1-q_\epsilon)+ u\over\beta_\epsilon},
 \end{align*}
 satisfies
\begin{align*}
u&= \ln \Big({r_\epsilon-1\over 1-q_\epsilon}+\Phi_\epsilon(u)\Big)+\int^{1}_{1-(1-q_\epsilon)\Phi_\epsilon(u)}{\phi^{(2)}_\epsilon(r_\epsilon,x)dx\over \phi_\epsilon(x)}+\int^{1}_{1-(1-q_\epsilon)\Phi_\epsilon(u)}{\phi(r_\epsilon)dx\over (x-q_\epsilon)\phi_\epsilon(x)}.
 \end{align*}
Letting $\epsilon\to0$ and using the standard tightness argument based on Helly's selection theorem,
we see that $\Phi_\epsilon(u)\to\Phi(u)$, as the limit distribution is uniquely determined by the equation \eqref{u}.
\end{proof}

\begin{theorem}\label{rop}
Consider a family of $(q_\epsilon,r_\epsilon)$-extendable branching processes satisfying \eqref{epo} with $m_1=1$, and assume \eqref{d}.

(i) If $d=0$,
then for any fixed $t\ge0$,
\[
{\rm P}\Big(T_{1,\epsilon}\le {t\over (r_\epsilon-1)m_2}\,\Big|\,T_{1,\epsilon}<\infty\Big)
\to1-e^{-t}. 
\]

(ii) If $d\in(0,\infty)$,
then for any fixed $t\ge0$,
\[
{\rm P}\Big(T_{1,\epsilon}\le {t\over a\sqrt{\epsilon}}\,\Big|\,T_{1,\epsilon}<\infty\Big)
\to{e^{t}-1\over e^{t}+d},\quad a=\sqrt{m_2(d+1/d)}. 
\]

(iii) If $d=\infty$,
then there is convergence to the standard logistic distribution
\[
 {\rm P}\Big(T_{1,\epsilon}\le{\ln d_\epsilon\over\beta_\epsilon}+{u\over \alpha_\epsilon}\,\Big|\,T_{1,\epsilon}<\infty\Big)\to{1\over1+e^{- u}},\quad u\in(-\infty,\infty).
\]

 \end{theorem}

\begin{proof} 
Items (i) and (ii) are obtained in the same way as Theorem \ref{rob} (i) using Lemma \ref{near}.
To prove (iii) we turn to \eqref{eto} and find that uniformly over $t\ge0$,
\[
 1-V_\epsilon(t)\sim e^{-\alpha_\epsilon t} \Big[1+d_\epsilon V_\epsilon(t)\Big]^{\gamma_\epsilon}\sim e^{-\alpha_\epsilon t} (d_\epsilon)^{\gamma_\epsilon}
 ,\quad \epsilon\to0.
\]
Choosing here $t={u+\gamma_\epsilon\ln d_\epsilon\over\alpha_\epsilon}$ we obtain statement (iii). Notice, that condition \eqref{epo} implies $m_2\in(0,\infty)$.
\end{proof}

\noindent {\bf Example.}
In the framework  of modified linear-fractional generating functions, condition \eqref{epo} requiring convergence over $s\in[0,s_0]$ with $s_0>1$, comes naturally in the form
\[
 f_\epsilon(s)\to p_0+p_1s+(1-p_0-p_1)s^2(1-p)(1-ps)^{-1},\quad 
 s\in[0,1/p).
\]
In this particular case the limit equation in
Proposition \ref{rob} (ii)  simplifies taking the form
\[e^{-u}\Phi(u)=(1-\Phi(u))^{1/\gamma}.\]
 For example, with $\gamma=1/2$, we get
 \[\Phi(u)
 =1-e^{-u}\sqrt{e^{u}-1/4}.\]
If $\gamma\in(0,1)$, then as $u\to\infty$
 \[1-\Phi(u)\sim e^{-\gamma u},\qquad \Phi(-u)\sim e^{-u}.\]

\begin{remark}
 Comparing these five asymptotic formulas for conditional distribution of the termination time, we find that the largest typical values are expected in the balanced near critical case with $d=1$, when
 $1-q_\epsilon\sim r_\epsilon-1$, 
and 
\[
{\rm P}\Big(T_{1,\epsilon}\le {t\over \sqrt{\epsilon}}\,\Big|\,T_{1,\epsilon}<\infty\Big)
\to{e^{at}-1\over e^{at}+1},\quad a=\sqrt{f''(1)}. 
\]
If a particle terminates the whole branching process with probability $\epsilon=10^{-4}$, then in the balanced nearly critical case with $d=1$ and $m_2=1$, this process does not go extinct  with aproximate probability  $\sqrt\epsilon=10^{-2}$. Conditioned on non-extinction, the process will terminate after a time of order 100 seconds (assuming that the average lifelength of a particle is one second).
\end{remark}

\section{A refined asymptotic formula  in the critical case}\label{Scri}

\begin{proposition}\label{prc} 
Consider a critical branching process satisfying  \eqref{xxx}.  Then for any fixed $s\in[0,1)$,
 \[1-F_t(s)={1\over m_2\, t}- {m_3\ln t\over m_2^3 \, t^2}- {A(s)\over m_2^2\,  t^2}
 +o\Big({1\over t^2}\Big),\quad  t\to\infty,\]
 where
 \[A(s)={1\over 1-s}-{m_3\over m_2}\ln(1-s)-m_2\int_{s}^1\psi_{1,1}(x)dx.
\]

\end{proposition}
 
\begin{proof}
  By Theorem \ref{main0}, for a given $s\in[0,1)$,
\begin{align*}
{1\over 1-F_t(s)}-{1\over 1-s}=m_2 t  +{m_3\over m_2} \ln {1-F_t(s)\over 1-s}
-m_2\int_{s}^{F_t(s)}\psi_{1,1}(x)dx.
 \end{align*}
It follows immediately that $1-F_t(s)\sim{1\over m_2t}$ as $t\to\infty$. Assuming \eqref{xxx} and applying Proposition \ref{pipi1} we obtain from the previous equality
\begin{align*}
{1\over 1-F_t(s)}=m_2 t +{m_3\over m_2} \ln(1-F_t(s))+A(s)+o(1).
 \end{align*}
 From here, using relation
 \[1-F_t(s)={1\over m_2t}\Big(1- {m_3\over m_2^2 }{\ln t\over t}-  {A(s)+\epsilon _t(s)\over m_2t}\Big)\]
 as the definition of  $\epsilon _t(s)$, we find that $\epsilon _t(s)=o(1)$, which is what had to be proven.
 \end{proof}

Plugging in the just proven formula $s=0$ we obtain the following asymptotic result for the probability of survival by time $t$.
\begin{corollary}\label{c21}
  If $f(1)=f'(1)=1$ and  \eqref{xxx} holds,  then 
 \[{\rm P}(Z_t>0)={1\over m_2 \, t}- {m_3\over m_2^3 }\,{\ln t\over t^2}- {A(0)\over m_2^2\,t^2}
 +o\Big({1\over t^2}\Big),\quad  t\to\infty.\]
\end{corollary}
 
 \begin{remark}
The last asymptotic formula should be compared to a formula on page 248 in \cite{Z}:
 \[{\rm P}(Z_t>0)={1\over m_2\, t}+{m_3\over m_2^3}\,{\ln t\over t^2}+ o\Big({\ln t\over t^2}\Big),\quad  t\to\infty.\]
Our formula provides with an expression for a higher order term, and also removes a misprint in Zolotarev's formula affecting the sign of the second term. (For a detailed account on the critical Markov branching processes under weaker moment conditions see \cite{Pa}.)

\end{remark}

\begin{corollary}\label{Co}
  If $f(1)=f'(1)=1$ and  \eqref{xxx} holds,  then for any $k\ge1$,
 \[{\rm P}(Z_t=k)\sim 
 {h_k\over t^2},\quad  t\to\infty,\]
with the sequence $\{h_k\}$ being characterised by
 \[\sum_{k=1}^\infty h_ks^k={1\over m_2^2}\, {s\over 1-s}+{m_3\over m_2^3}\,\ln{1\over 1-s}+{1\over m_2^2}\,\int_0^s\psi_{1,1}(x)dx.
\]
\end{corollary}
\begin{proof}
The statement follows from
\[\sum_{k=1}^\infty {\rm P}(Z_t=k)s^k=F_t(s)-F_t(0)={A(s)-A(0)\over m_2^2\, t^2}
 +o\Big({1\over t^2}\Big),\quad  t\to\infty.\]
\end{proof}

\noindent {\bf Example 1.}
If $f$ is given by \eqref{ecr}, then by  Corollary \ref{c21},
 \[{\rm P}(Z_t>0)={1-p\over p_0 \,  t}- {(1-p)p\over p_0^2 }\, {\ln t\over t^2}- {(1-p)^2\over p_0^2\,  t^2}
 +o\Big({1\over t^2}\Big),\]
 as $ t\to\infty$, and by Corollary \ref{Co},
  \[t^2{\rm P}(Z_t=k)\to {1-p\over p_0^2}\Big (1-p{k-1\over k}\Big),\quad k\ge1.\]

\noindent {\bf Example 2.}
Consider the critical case in the framework of Section \ref{Ex}:
\[f(s)=s+{1\over2}(1-s)^2(2p_3s+1+p_3),\quad p_3\in[0,1).\]
 By Proposition \ref{prc}, 
\begin{align*}
 1-F_t(s)&={2\over 1+3p_3}\,{1\over t}- {8p_3\over (1+3p_3)^3 }\,{\ln t\over t^2}-{4A(s)\over (1+3p_3)^2\,  t^2}
 +o\Big({1\over t^2}\Big),\quad t\to\infty,\\
   A(s)&={1\over 1-s}+{2p_3\over 1+3p_3}\,\ln{2(1+p_3+2p_3s)\over (1-s)(1+3p_3)},
\end{align*}
and by Corollary \ref{c21},
\[ {\rm P}(Z_t>0)={2\over 1+3p_3}\,{1\over t}- {8p_3\over (1+3p_3)^3 }\,{\ln t\over t^2}-{4(1+3p_3)+8\ln{2(1+p_3)\over 1+3p_3}\over (1+3p_3)^3}\,   {1\over t^2}+o\Big({1\over t^2}\Big),\quad t\to\infty.\]
Corollary \ref{Co} holds with
\begin{align*}
 \sum_{k=1}^\infty h_ks^k&={4\over (1+3p_3)^2}\Big( {s\over 1-s}+{2p_3\over 1+3p_3}\ln{1+p_3+2p_3s\over (1-s)(1+p_3)}\Big).
\end{align*}

\section{A new proof  in the supercritical case
}\label{Ssup}

Given $f(1)=1$, the branching process normalised by its mean $M_t=e^{-(m_1-1)t}$ forms a non-negative martingale implying almost sure  convergence
 $Z_t/M_t\to W,\quad t\to\infty.$
Thus for  $\rho\ge0$, 
\begin{equation*}
{\rm E}e^{-\rho Z_t/m_t}\to {\rm E}e^{-\rho W},\quad t\to\infty.
\end{equation*}
In this section we apply the tail generating function technique to give streamlined proofs for classical results concerning the limit Laplace transform $\eta(\rho)={\rm E}e^{-\rho W}$.

\begin{proposition}\label{gsup}
Let $f(1)=1$ and $m_1\in(1,\infty)$. If
\begin{equation}\label{xlx}
 \sum_{k=2}^\infty p_kk\ln k<\infty,
\end{equation}
then for $\rho>0$, we have
\begin{align}\label{eqW}
 \eta(\rho)=q+(1-q)\Big( {1-\eta(\rho)\over\rho}\Big)^\gamma \exp\Big\{-\int_{\eta(\rho)}^1\psi_{q,1}(x)dx\Big\},
 \end{align}
so that   $\eta(\rho)\in(q,1)$. If \eqref{xlx} does not hold, then $\eta(\rho)=1$, $\rho\ge0$, so that ${\rm P}(W=0)=1$.
\end{proposition}
\begin{proof} 
%
By Theorem \ref{main} with $r=1$, in view of $M_t=e^{-\beta t}$, we have
 \[{F_t(s)-q\over s-q}= \Big\{M_t\,{1-F_t(s)\over 1-s}\Big\}^\gamma \exp\Big\{\int_{s}^{F_t(s)}\psi_{q,1}(x)dx\Big\}.\]
Replacing $s$ with  $s_t=e^{-\rho/M_t}$ and putting $\eta_t(\rho)=F_t(s_t)$,  we obtain
\begin{align*}
 {\eta_t(\rho)-q\over s_t-q}
\sim \Big( {1-\eta_t(\rho)\over\rho}\Big)^\gamma \exp\Big\{\int_{s_t}^{\eta_t(\rho)}\psi_{q,1}(x)dx\Big\},\quad t\to\infty.
 \end{align*}
 Thus, if \eqref{xlx} holds, then due to Proposition \ref{pipi4} with $r=1$, equation \eqref{eqW} follows, which in turn implies that $\eta(\rho)\in(q,1)$ for $\rho>0$. 
 
 On the other hand,  if \eqref{xlx} does not hold, then again by Proposition \ref{pipi4},
\begin{align*}
\exp\Big\{\int_{s_t}^{\eta_t(\rho)}\psi_{q,1}(x)dx\Big\}={ \mathcal L_{q,1}(1-\eta_t(\rho))\over \mathcal L_{q,1}(1-s_t)},
 \end{align*}
where $\mathcal L_{q,1}(1-s_t)\to0$ as $t\to\infty$. We conclude that in this case $\eta_t(\rho)\to 1$.
\end{proof}

\begin{corollary}
 If \eqref{xlx} holds, then ${\rm E}W=1$ and there is a positive constant $C$ such that 
 $${\rm P}(W\le t|W>0)\sim Ct^{\gamma},\quad t\to0.$$ 
 Also, if $m_2<\infty$, then ${\rm E}W^2={2m_2\over m_1-1}$.
 
\end{corollary}
\begin{proof}
In view of \[{\rm E}W^n=(-1)^n\eta^{(n)}(0),\quad n\ge1,\]
equation \eqref{eqW} implies ${\rm E}W=1$. Furthermore, by a Taylor expansion 
we have
\[{1-\eta(\rho)\over\rho}=1-{\eta''(0)\over2}\rho+o(\rho),\]
which together with \eqref{eqW} and \eqref{qqrr} give
\begin{align*}
{\gamma\eta''(0)\over2}=-{ 1\over1-q}-\psi_{q,1}(1)={ \gamma m_2\over\beta}.
 \end{align*}
 Thus, we obtain
${\rm E}W^2={2m_2\over \beta}$.
  Next, observe that  as $\rho\to\infty$, equation \eqref{eqW} gives
 \[  \rho^\gamma (\eta(\rho)-q)\to(1-q)q^\gamma \exp\Big\{-\int_q^1\psi_{q,1}(x)dx\Big\},
\]
implying 
$${\rm E}(e^{-\rho W}|W>0)={\eta(\rho)-q\over 1-q}\sim C_1\rho^{-\gamma},\quad \rho\to\infty,$$
 for some $C_1\in(0,\infty)$. By the Tauberian Theorem 2 from \cite[Ch. XIII.5]{F2} this brings the second claim of the corollary.
\end{proof}

\noindent {\bf Examples.}
For a modified linear-fractional $f$, equation \eqref{eqW} can be written as 
\[ \eta(\rho)=q+(1-q)\Big( {1-\eta(\rho)\over\rho}\Big)^\gamma.
\]
In particular, if $\gamma=1$, the limit distribution both is exponential, in that
 \[{\rm E}(e^{-\rho W}|W>0)={1-q\over 1-q+\rho}.\]
If $\gamma={1\over 2}$, then
\begin{align*}
 \eta(\rho)
 =q+(1-q){\sqrt{(1-q)(1-q+4\rho)}-1+q\over 2\rho}.
 \end{align*}
For the example from Section \ref{Ex}, as compared to the previous example, we obtain an extra term in the equation
\begin{align*}
 \eta(\rho)=q+ (1-q)\Big( {1-\eta(\rho)\over\rho}\Big)^\gamma \Big({1+p_3qr+(q+r)p_3\eta(\rho)\over 1+p_3q+p_3r+p_3qr}\Big)^{1-\gamma}.
 \end{align*}
 
\begin{remark}
Equation \eqref{eqW}  should be compared to its counterpart stated in Theorem 3 from \cite[Ch III.8]{AN}, which can be rewritten as
\begin{equation}\label{iro}
  \eta(\rho)=1-\rho\exp\Big\{\int_{\eta(\rho)}^1\Big({m_1-1\over f(x)-x}+{1\over 1-x}\Big)dx\Big\}.
\end{equation}
To demonstrate equivalence of these two equations we use the chain of equalities
\begin{align*}
 {m_1-1\over f(x)-x}+{1\over 1-x}&={1\over(1-x)(q-x)\phi(x)}(f^{(2)}(1,1)-f^{(2)}(1,x))
 ={f^{(3)}(1,1,x)\over(q-x)\phi(x)}\\
&={f^{(3)}(q,1,1)\over(q-x)\phi(x)}-{f^{(4)}(q,1,1,x)\over\phi(x)}={\phi(1)\over(q-x)\phi(q)}+{\phi(1)\phi^{(2)}(q,x)\over\phi(q)\phi(x)}-{\phi^{(2)}(1,x)\over\phi(x)}.
\end{align*}
Since ${\phi(1)\over\phi(q)}=\gamma^{-1}$, we arrive at \eqref{eqW} after observing that according to \eqref{iro}
\[\gamma\ln {1-\eta(\rho)\over\rho}=\gamma\int_{\eta(\rho)}^1\Big({m_1-1\over f(x)-x}+{1\over 1-x}\Big)dx=\ln{\eta(\rho)-q\over 1-q}+\int_{\eta(\rho)}^1
\psi_{q,1}(x)dx.\]
\end{remark}

\end{document}